\documentclass{article}%
\usepackage{amsfonts}
\usepackage{amsmath}
\usepackage{amssymb}
\usepackage{graphicx}%
\providecommand{\U}[1]{\protect\rule{.1in}{.1in}}
\newtheorem{theorem}{Theorem}

\newtheorem{corollary}[theorem]{Corollary}

\newtheorem{example}[theorem]{Example}

\newtheorem{lemma}[theorem]{Lemma}

\newtheorem{proposition}[theorem]{Proposition}
\newtheorem{remark}[theorem]{Remark}

\newenvironment{proof}[1][Proof]{\noindent\textbf{#1.} }{\ \rule{0.5em}{0.5em}}
\begin{document}

\title{\textbf{Equivariant prequantization and the moment map}}
\date{}
\author{\textsc{Roberto Ferreiro P\'{e}rez}\\Departamento de Econom\'{\i}a Financiera y Actuarial y Estad\'{\i}stica\\Facultad de Ciencias Econ\'omicas y Empresariales, UCM\\Campus de Somosaguas, 28223-Pozuelo de Alarc\'on, Spain\\\emph{E-mail:} \texttt{roferreiro@ccee.ucm.es}}
\maketitle

\begin{abstract}
If $\omega$ is a closed $G$-invariant $2$-form and $\mu$ a moment map, then we
obtain necessary and sufficient conditions for equivariant pre-quantizability
that can be computed in terms of the moment map $\mu$. We also compute the
obstructions to lift the action of $G$ to a pre-quantization bundle of
$\omega$. Our results are valid for any compact and connected Lie group $G$.\ \ \ \ \ \ \ \ \ \ \ \ \ \ \ \ \ \ \ \ \ \ \ \ \ \ \ \ \ \ \ 

\end{abstract}

\noindent\emph{Acknowledgments:\/} Supported by Ministerio de Ciencia,
Innovaci\'{o}n y Universidades of Spain under grant PGC2018-098321-B-I00.

\section{Introduction}

In Geometric Quantization, a closed $2$-form $\omega\in\Omega_{G}^{2}(M)$ on a
manifold $M$ is said to be pre-quantizable if there exists a principal
$U(1)$-bundle $P\rightarrow M$ with connection $\Xi$ such that $\mathrm{curv}%
(\Xi)=\omega$. A classical result of Weil and Kostant (e.g. see \cite{Kostant}%
) states that $\omega$ is pre-quantizable if and only if $\omega$ is integral,
i.e., its cohomology class lies in the image of the map $H^{2}(M,\mathbb{Z}%
)\rightarrow H^{2}(M,\mathbb{R})$.

If a group of symmetries $G$ acts on $M$ and $\omega$ is $G$-invariant, it is
natural to ask if it is possible to pre-quantize $\omega$ mantaining the
symmetry. In more detail, we say that a $G$-invariant $2$-form $\omega$ is
$G$-equivariant pre-quantizable if there exists a lift of the action to a
principal $U(1)$-bundle $P\rightarrow M$ with connection $\Xi$ such that $\Xi$
is $G$-invariant and $\mathrm{curv}(\Xi)=\omega$. In this case, the moment
$\mathrm{Mom}^{\Xi}$\ of $\Xi$ gives a moment map for $\omega$. In particular,
a necessary condition for $\omega$ to be $G$-equivariant pre-quantizable is
that the action of $G$ should be Hamiltonian. If $\mu$ is a moment map for
$\omega$, we say that $(\omega,\mu)$ is $G$-equivariant pre-quantizable\ if
there exists a $G$-equivariant $U(1)$-bundle with connection $(P,\Xi)$ such
that $\mathrm{curv}(\Xi)=\omega$ and $\mathrm{Mom}^{\Xi}=\mu$. For a compact
group $G$, a necessary and sufficient condition for $G$-equivariant
pre-quantizability can be obtained in terms of $G$-equivariant cohomology of
$M$ (e.g. see \cite{GGK},\cite{Mundet}). The problem is that this condition
cannot be computed directly in terms of $\omega$ and $\mu$. In the case in
which $G$ is a torus and $\omega$ is symplectic, there is another
characterization of $G$-equivariant pre-quantizability that can be expressed
directly in terms of $\omega$ and $\mu$ (e.g. see \cite[Example 6.10]{GGK}).
Precisely, $(\omega,\mu)$ is $G$-equivariant pre-quantizable if and only if
$\omega$ is integral and for a fixed point $x\in M^{G}$ we have $\mu_{X}%
(x)\in\mathbb{Z}$ for any $X\in\ker(\exp_{G})$.

In this paper we generalize the preceding result to the case of non-abelian
groups by using the concept of equivariant holonomy. We show that, given a
pre-quantization bundle $(P,\Xi)$ for $\omega$, the obstruction to lift the
action to $P$ is given by a group homomorphism $\Lambda^{\Xi,\mu}\colon
H_{1}(G,\mathbb{Z)\rightarrow\mathbb{R}}/\mathbb{\mathbb{Z}}$. Under mild
assumptions, this obstruction is equivalent to a map $\Delta^{\Xi,\mu}%
\colon\ker(\exp_{G})\mathbb{\rightarrow\mathbb{R}}/\mathbb{\mathbb{Z}}$ that
can be computed in terms of the moment map $\mu$. Moreover, if $G$ is compact
and we do not fix $\mu$, then the obstruction is given by the restriction of
$\Lambda^{\Xi,\mu}$ to the torsion subgroup of $H_{1}(G,\mathbb{Z)}$, and we
have a similar results for the restriction of $\Delta^{\Xi,\mu}$ to a subset
$\mathfrak{T}_{G}\subset\ker(\exp_{G})$. In the important case of a compact
symplectic manifold we obtain that $(\omega,\mu)$ is $G$-equivariant
pre-quantizable if and only if $\omega$ is integral and $\max_{M}(\mu_{X}%
)\in\mathbb{Z}$ for any $X\in\ker(\exp_{G})$. Furthermore, $\omega$ is
$G$-equivariant pre-quantizable if an only if $\omega$ is integral and
$\max_{M}(\mu_{X})\in\mathbb{Z}$ for any $X\in\mathfrak{T}_{G}$.

\section{Equivariant holonomy}

In this section we recall the definition and properties of equivariant
holonomy introduced in \cite{AnomaliesG} and \cite{EquiHol}. Let $G$ be a
connected Lie group with Lie algebra $\mathfrak{g}$ and let $M$ be a connected
and oriented manifold. A $G$-equivariant $U(1)$-bundle is a principal
$U(1)$-bundle $P\rightarrow M$ in which $G$ acts (on the left) by principal
bundle automorphisms. If $\phi\in G$ and $y\in P$, we denote by $\phi_{P}(y)$
the action of $\phi$ on $y$. In a similar way, for $X\in\mathfrak{g}$ we
denote by $X_{P}\in\mathfrak{X}(P)$ the corresponding vector field on $P$
defined by $X_{P}(x)=\left.  \frac{d}{dt}\right\vert _{t=0}\exp(-tX)_{P}(x)$.

We denote by $I$ the interval $[0,1]$. If $\gamma\colon I\rightarrow M$ is a
curve, we define the inverse curve $\overleftarrow{\gamma}\colon I\rightarrow
M$ by $\overleftarrow{\gamma}(t)=\gamma(1-t)$. Moreover, if $\gamma_{1}$ and
$\gamma_{2}\mathcal{\ }$are curves with $\gamma_{1}(1)=\gamma_{2}(0)$ we
define $\gamma_{1}\ast\gamma_{2}\colon I\rightarrow\mathbb{R}$ by $\gamma
_{1}\ast\gamma_{2}(t)=\gamma_{1}(2t)$ for $t\in\lbrack0,1/2]$ and $\gamma
_{1}\ast\gamma_{2}(t)=\gamma_{2}(2t-1)$ for $t\in\lbrack1/2,1]$. For any
$\phi\in G$ we define $\mathcal{C}^{\phi}(M)=\{\gamma\colon I\rightarrow M$
$|$ $\gamma$ is piecewise smooth and $\gamma(1)=\phi_{M}(\gamma(0))\}$, and
$\mathcal{C}_{x}^{\phi}(M)=\{\gamma\in\mathcal{C}^{\phi}(M)$ $|$
$\gamma(0)=x\}$. Note that if $e\in G$ is the identity element, then
$\mathcal{C}_{x}^{e}(M)=\mathcal{C}_{x}(M)$ is the space of loops based at
$x$. If $\phi\in G$ and $\gamma$ is a curve on $M$ then we define $\phi
\cdot\gamma$ by $(\phi\cdot\gamma)(t)=\phi_{M}(\gamma(t))$.

Let $\Xi$ be a $G$-invariant connection on a $G$-equivariant $U(1)$-bundle
$P\rightarrow M$. If $\phi\in G$ and $\gamma\in\mathcal{C}^{\phi}(M)$, the
$\phi$-equivariant holonomy $\mathrm{hol}_{\phi}^{\Xi}(\gamma)\in U(1)$ of
$\gamma$ is characterized by the property $\overline{\gamma}(1)=\phi
_{P}(\overline{\gamma}(0))\cdot\exp(2\pi i\mathrm{hol}_{\phi}^{\Xi}(\gamma))$
for any $\Xi$-horizontal lift $\overline{\gamma}\colon I\rightarrow P$ of
$\gamma$. Note that if $\gamma\in\mathcal{C}_{x}^{e}(M)$ is a loop on $M$,
then $\mathrm{hol}_{e}^{\Xi}(\gamma)$ is the ordinary holonomy of $\gamma$.
The following result is proved in \cite{EquiHol}

\begin{proposition}
\label{PropHol}If $P\rightarrow M$ is a $G$-equivariant principal
$U(1)$-bundle, and $\Xi$ is a $G$-invariant connection on $P$, then for any
$\phi$, $\phi^{\prime}\in G$\ we have

a) If $\gamma\in\mathcal{C}^{\phi}(M)$ and $\gamma^{\prime}\in\mathcal{C}%
_{\gamma(1)}^{\phi^{\prime}}(M)$, then $\gamma\ast\gamma^{\prime}%
\in\mathcal{C}^{\phi^{\prime}\cdot\phi}(M)$ and we have $\mathrm{hol}%
_{\phi^{\prime}\cdot\phi}^{\Xi}(\gamma\ast\gamma^{\prime})=\mathrm{hol}_{\phi
}^{\Xi}(\gamma)+\mathrm{hol}_{\phi^{\prime}}^{\Xi}(\gamma^{\prime})$ and
$\mathrm{hol}_{\phi^{-1}}^{\Xi}(\overleftarrow{\gamma})=-\mathrm{hol}_{\phi
}^{\Xi}(\gamma)$.

b) If $\gamma\in\mathcal{C}^{\phi}(M)$ and $\zeta$ is a curve on $M$ such that
$\zeta(0)=\gamma(0)$\ then we have $\mathrm{hol}_{\phi}^{\Xi}(\overleftarrow
{\zeta}\ast\gamma\ast(\phi\cdot\zeta))=\mathrm{hol}_{\phi}^{\Xi}(\gamma)$.
\end{proposition}

If $\Xi$ is a $G$-invariant connection on a principal $U(1)$ bundle
$P\rightarrow M$ then $\frac{i}{2\pi}d\Xi$ projects onto a closed
$G$-equivariant $2$-form $\mathrm{curv}(\Xi)\in\Omega^{2}(M)$ called the
curvature of $\Xi$. We have the following generalization of the classical
Gauss-Bonnet Theorem for surfaces

\begin{proposition}
\label{HolonomyInt}If $\Sigma\subset M$ is a $2$-dimensional submanifold with
boundary $\partial\Sigma=\bigcup\limits_{i=1}^{k}\gamma_{i}$, with $\gamma
_{i}\in\mathcal{C}(M)$ then\ we have $%
{\textstyle\sum_{i=1}^{k}}
\mathrm{hol}^{\Xi}(\gamma_{i})=\int_{\Sigma}\mathrm{curv}(\Xi
)\operatorname{mod}\mathbb{Z}$.
\end{proposition}

If $\omega\in\Omega^{2}(M)$ is a closed $G$-invariant $2$-form, we say that
the action is Hamiltonian if there exist a moment map for $\omega$, i.e., if
there exists a $G$-equivariant map $\mu\colon\mathfrak{g}\rightarrow\Omega
^{0}(M)$ such that $i_{X}\omega=d(\mu(X))$ for any $X\in\mathfrak{g}$. In the
case of a $G$-invariant connection $\Xi$, its curvature $\mathrm{curv}(\Xi)$
has a moment map $\mathrm{Mom}^{\Xi}\colon\mathfrak{g}\rightarrow\Omega
^{0}(M)$ defined by $\mathrm{Mom}_{X}^{\Xi}(x)=-\frac{i}{2\pi}\Xi(X_{P}(y))$
for any $y\in P$ such that $\pi(y)=x$. Furthermore, we have the following
result (see \cite{EquiHol})

\begin{proposition}
\label{InfHolonomy}For any $X\in\mathfrak{g}$ and $x\in M$ we define
$\tau_{x,X}(s)=\exp(sX)_{M}(x)$. Then $\tau_{x,X}\in\mathcal{C}_{x}^{\exp
(X)}(M)$\ and we have $\mathrm{hol}_{\exp(X)}^{\Xi}(\tau_{x,X})=\mathrm{Mom}%
_{X}^{\Xi}(x)\operatorname{mod}\mathbb{Z}$.
\end{proposition}

\section{Lifting the action to a pre-quantization bundle\label{Sect Lift}}

Let $\omega$ be a closed $G$-invariant $2$-form and let $(P,\Xi)$ be a
(non-equivariant) pre-quantization of $\omega$. A classical problem in
Symplectic Geometry is if the action of $G$ can be lifted to $P$ leaving $\Xi
$\ invariant (e.g. see \cite{Gotay},\cite{Kostant},\cite{Mundet} and
references therein). If $\mu$ is a moment map for $\omega$, we can also impose
that $\mathrm{Mom}^{\Xi}=\mu$. We study these problems in detail in the next Sections.

\subsection{Lifting the action with the moment map fixed}

Let $P\rightarrow M$ be a principal $U(1)$-bundle, $\Xi$ a connection on $P$
and let $\mu\colon\mathfrak{g}\rightarrow\Omega^{0}(M)$ be a moment map for
$\omega=\mathrm{curv}(\Xi)$. We say that the action of $G$ on $M$ admits a
$(\Xi,\mu)$-lift to $P$ if there exists a lift of the action of $G$ to
$P\rightarrow M$ by automorphisms such that $\Xi$ is $G$-invariant and
$\mathrm{Mom}^{\Xi}=\mu.$

At the infinitesimal level it is possible to lift the action of $\mathfrak{g}$
to $P$ in such a way that $L_{X_{P}}\Xi=0$ and $\iota_{X_{P}}\Xi=-\mu_{X}$ for
any $X\in\mathfrak{g}$ (e.g. see \cite{GGK},\cite{Kostant},\cite{Mundet}).
Precisely, the lift is defined by setting $X_{P}(y)=H_{y}^{\Xi}(X_{N}%
(x))-\mu_{X}(x)\xi_{P}(y)$, where $x=\pi(y)$, $H_{y}^{\Xi}(X_{N})$ is the
$\Xi$-horizontal lift of $X_{N}$\ and $\xi_{P}$ the vector field associated to
the action of $U(1)$ on $P$. If $G$ is connected and simply connected, then
the Lie algebra action can be exponentiated to an action of $G$ on $P$ (e.g.
see \cite{Kostant}) and hence the action of $G$ admits a $(\Xi,\mu)$-lift to
$P$. However, if $G$ is connected but not simply connected there could be
obstructions to lift the action of $G$ to $P$ (e.g. see \cite{EquiHol}%
,\cite{Mundet}).

Let $\rho_{G}\colon\tilde{G}\rightarrow G$ be the universal covering group of
$G$. The action of $G$ induces an action of $\tilde{G}$ on $M$ and, as
$\tilde{G}$ is simply connected, this action admits a $(\Xi,\mu)$-lift to $P$,
i.e., we have a group homomorphisms $\beta\colon\tilde{G}\rightarrow
\mathrm{Aut}P$. The homomorphism $\beta$\ defines a group homomorphism
$G\rightarrow\mathrm{Aut}P$ if and only if $\beta(\ker\rho_{G})=\mathrm{id}%
_{P}$, i.e., if and only if $\phi_{P}=\mathrm{id}_{P}$ for any $\phi\in
\ker\rho_{G}$. The group $\ker\rho_{G}$ can be identified with the fundamental
group $\pi_{1}(G)$ that it is well known to be an abelian group and hence we
have $\ker\rho_{G}\simeq\pi_{1}(G)\simeq H_{1}(G,\mathbb{Z})$. The explicit
isomorphism is given by the map that assigns to $\phi\in\ker\rho_{G}$ the
class of $\rho_{G}(\gamma)$, where $\gamma$ is any curve on $\tilde{G}$
joining $1_{G}$ and $\phi$.

If $\phi\in\ker\rho_{G}$ and $y\in P$ then we have $\phi_{P}(y)=y\cdot
\exp(-2\pi i\cdot\Lambda_{y}^{\Xi,\mu}(\phi))$ for certain $\Lambda_{y}%
^{\Xi,\mu}(\phi)\in\mathbb{R}/\mathbb{Z}$. The action admits a $(\Xi,\mu
)$-lift to $P$ if and only if\ $\Lambda_{y}^{\Xi,\mu}(\phi
)=0\operatorname{mod}\mathbb{Z}$ for any $\phi\in\ker\rho_{G}\simeq
H_{1}(G,\mathbb{Z})$ and $y\in P$. We study this condition in terms of the
$\tilde{G}$-equivariant holonomy of $\Xi$. From the definition of equivariant
holonomy we obtain the following

\begin{proposition}
\label{hol}If $\phi\in\ker\rho_{G}$ then for any $\gamma\in\mathcal{C}%
_{\pi(y)}^{\phi}(M)=\mathcal{C}_{\pi(y)}^{e}(M)$ we have $\Lambda_{y}^{\Xi
,\mu}=\mathrm{hol}_{\phi}^{\Xi}(\gamma)-\mathrm{hol}_{e}^{\Xi}(\gamma)$.
\end{proposition}

In particular, by applying the preceding proposition to the constant curve
$c_{x}$, with $x=\pi(y)$ we obtain $\Lambda_{y}^{\Xi,\mu}=\mathrm{hol}_{\phi
}^{\Xi}(c_{x})$.

\begin{proposition}
$\Lambda_{y}^{\Xi,\mu}(\phi)$ is a constant function on $P$ (i.e., it does not
depend on $y\in P$).
\end{proposition}

\begin{proof}
If $y^{\prime}\in P$ is another point and $x^{\prime}=\pi(x)$, we choose a
curve $\zeta$ with $\zeta(0)=x$ and $\zeta(1)=x^{\prime}$. As $\phi\in\ker
\rho$ we have $(\phi\cdot\zeta)=\zeta$\ and hence
\begin{align*}
0  &  =\mathrm{hol}_{e}^{\Xi}(\overleftarrow{\zeta}\ast c_{x}\ast\zeta
\ast\overleftarrow{c}_{x^{\prime}})=\mathrm{hol}_{\phi}^{\Xi}(\overleftarrow
{\zeta}\ast c_{x}\ast\zeta)+\mathrm{hol}_{\phi^{-1}}^{\Xi}(\overleftarrow
{c}_{x^{\prime}})\\
&  =\mathrm{hol}_{\phi}^{\Xi}(c_{x})-\mathrm{hol}_{\phi}^{\Xi}(c_{x^{\prime}%
}).
\end{align*}

\end{proof}

We denote $\Lambda_{y}^{\Xi,\mu}$\ simply by $\Lambda^{\Xi,\mu}$, and we have
a well defined group homomorphism $\Lambda^{\Xi,\mu}\colon\ker\rho
_{G}\rightarrow\mathbb{R}/\mathbb{Z}$ and the following

\begin{proposition}
The action of $G$ on $M$ admits a $(\Xi,\mu)$-lift if and only if the
homomorphism $\Lambda^{\Xi,\mu}\colon\ker\rho_{G}\simeq H_{1}(G,\mathbb{Z}%
)\rightarrow\mathbb{R}/\mathbb{Z}$ vanishes.
\end{proposition}

We can obtain an equivalent conditions in terms of Lie algebra $\mathfrak{g}$
of $G$. Let $\exp_{G}\colon\mathfrak{g}\rightarrow G$ and $\exp_{\tilde{G}%
}\colon\mathfrak{g}\rightarrow\tilde{G}$ be the exponential maps and define
$\ker\exp_{G}=\{X\in\mathfrak{g}:\exp(X)=1_{G}\}$. We have $\exp_{\tilde{G}%
}(\ker\exp_{G})\subset\ker\rho_{G}$, and we define $\Delta^{\Xi,\mu}\colon
\ker\exp_{G}\rightarrow\mathbb{R}/\mathbb{Z}$ by $\Delta^{\Xi,\mu}%
(X)=\Lambda^{\Xi,\mu}(\exp_{\tilde{G}}X)$. By applying Propositions
\ref{InfHolonomy}\ and \ref{hol} to the curve $\tau_{x,X}$\ we obtain the following

\begin{proposition}
\label{calculo}If $X\in\ker\exp_{G}$ then

a) for any $x\in M$ we have $\Delta^{\Xi,\mu}(X)=\mu_{X}(x)-\mathrm{hol}^{\Xi
}(\tau_{x,X})\operatorname{mod}\mathbb{Z}.$

e) If $x\in M$ and $X_{M}(x)=0$ then $\Delta^{\Xi,\mu}(X)=\mu_{X}%
(x)\operatorname{mod}\mathbb{Z}$.
\end{proposition}

In particular, if $x\in M^{G}$ is a fixed point for the action of $G$, then
$\Delta^{\Xi,\mu}(X)=\mu_{X}(x)\operatorname{mod}\mathbb{Z}$ for any $X\in
\ker\exp_{G}$.

\begin{corollary}
If $\omega$ is integral and $\mu$ is a moment map for $\omega$, then for any
$X\in\ker\exp_{G}$ and any $x,x^{\prime}\in M$ such that $X_{M}(x)=X_{M}%
(x^{\prime})=0$ we have $\mu_{X}(x)-\mu_{X}(x^{\prime})\in\mathbb{Z}$. In
particular, for any two fixed points $x,x^{\prime}\in M^{G}$ we have $\mu
_{X}(x)-\mu_{X}(x^{\prime})\in\mathbb{Z}$.
\end{corollary}

\begin{proposition}
If $\omega=\mathrm{curv}(\Xi)$ is symplectic (i.e. non-degenerated) then for
any $X\in\ker\exp_{G}$ and any critical point $x$ of $\mu_{X}$ we have
$\Delta^{\Xi,\mu}(X)=\mu_{X}(x)\operatorname{mod}\mathbb{Z}.$
\end{proposition}

\begin{proof}
We have $0=d\mu_{X}(x)=i_{X_{M}}\omega(x)$, and as $\omega$ is non-degenerate,
we conclude that $X_{M}(x)=0$. By Proposition \ref{calculo}\ we have
$\Delta^{\Xi,\mu}(X)=\mu_{X}(x)\operatorname{mod}\mathbb{Z}$.
\end{proof}

\begin{corollary}
If $M$ is compact and $\omega$ is symplectic then for any $X\in\ker\exp_{G}$
we have $\Delta^{\Xi,\mu}(X)=\max_{M}(\mu_{X})\operatorname{mod}%
\mathbb{Z}=\min_{M}(\mu_{X})\operatorname{mod}\mathbb{Z}$.
\end{corollary}

In particular we obtain the following

\begin{corollary}
If $\omega$ is symplectic and integral, and $M$ is compact, then for any
$X\in\ker\exp_{G}$ we have $\max_{M}(\mu_{X})-\min_{M}(\mu_{X})\in\mathbb{Z}$.
\end{corollary}

We recall that a Lie group is called exponential if the exponential map
$\exp_{G}\colon\mathfrak{g}\rightarrow G$ is surjective. For example, if $G$
is compact and connected, then $G$ and $\tilde{G}$\ are exponential (e.g. see
\cite[Proposition 6.10]{Helgason}). We say that $G$ is w-exponential if
$\ker\rho_{G}\subset\exp_{\tilde{G}}(\mathfrak{g})$. For example, any compact
group is w-exponential. For a w-exponential group $G$ we have $\exp_{\tilde
{G}}(\ker\exp_{G})=\ker\rho_{G}$ and hence it is equivalent to work with
$\Delta^{\Xi,\mu}$ or with $\Lambda^{\Xi,\mu}$.

\begin{theorem}
If $G$ is w-exponential, then the action of $G$ on $M$ admits a $(\Xi,\mu
)$-lift to $P$ if and only if $\Delta^{\Xi,\mu}(X)=0\operatorname{mod}%
\mathbb{Z}$ for any $X\in\ker\exp_{G}$.
\end{theorem}

\begin{example}
\label{Trivial}Let $M$ be a connected manifold in which $G=S^{1}$ acts. We set
$\omega=0$ and $\mu^{c}\colon\mathbb{R}\rightarrow\Omega^{0}(M)$ given by
$\mu_{X}^{c}(x)=cX$ for any $x\in M$ and $X\in\mathbb{R}$. Clearly $\omega$ is
$G$-invariant, $\mu^{c}$ is a moment map for $\omega$ and $\ker\exp_{S^{1}%
}=2\pi\mathbb{Z}$. Let $(P,\Xi)$ be the trivial bundle and connection. If
$X=2\pi z$ then we have $\Delta^{\Xi,\mu^{c}}(X)=\mu_{X}^{c}(x)-\mathrm{hol}%
^{\Xi}(\tau_{x,X})\operatorname{mod}\mathbb{Z}=2\pi zc\operatorname{mod}%
\mathbb{Z}$. We conclude that the action admits a $(\Xi,\mu)$-lift to $P$ if
and only if $2\pi c\in\mathbb{Z}$.
\end{example}

\begin{corollary}
If $G$ is w-exponential and $M^{G}\neq\emptyset$, then the action admits a
$(\Xi,\mu)$-lift to $P$ if and only if for a fixed point $x\in M^{G}$ we have
$\mu_{X}(x)\in\mathbb{Z}$ for any $X\in\ker\exp_{G}$.
\end{corollary}

If $G$ is a torus and $\omega$ is symplectic, the Atiyah-Guillemin-Sternberg
convexity Theorem implies the existence of fixed points, and hence the
preceding theorem can be applied.\ For non abelian groups we cannot assert
that $M^{G}\neq\emptyset$, but we have the following

\begin{corollary}
If $G$ is w-exponential, $M$ is compact and $\omega$ is symplectic, then the
action admits a $(\Xi,\mu)$-lift to $P$ if and only if for any $X\in\ker
\exp_{G}$ we have $\max_{M}(\mu_{X})\in\mathbb{Z}$.
\end{corollary}

\section{Lifting the action with an arbitrary moment map}

In this Section we assume that $G$ is a connected and compact Lie group. Let
$P\rightarrow M$ be a principal $U(1)$-bundle, $\Xi$ a connection on $P$ and
$\omega=\mathrm{curv}(\Xi)$. We say that the action of $G$ on $M$ admits a
$\Xi$-lift to $P$ if there exists a lift of the action of $G$ to $P\rightarrow
M$ by automorphisms such that $\Xi$ is $G$-invariant. In this case
$\mathrm{Mom}^{\Xi}$\ is a moment map for $\omega$ and hence, there exists a
$\Xi$-lift if and only if there exists a $(\Xi,\mu)$-lift for a moment map
$\mu$.

We assume that the action is Hamiltonian and that $\mu$ is a moment map for
$\omega$. As we have seen, the obstruction to the existence of a $(\Xi,\mu
)$-lift is given by $\Delta^{\Xi,\mu}\in\mathrm{Hom}(H_{1}(G,\mathbb{Z}%
),\mathbb{R}/\mathbb{Z})$. Any other moment map is given by $\mu^{\prime}%
=\mu+b$ where $b\in H^{1}(\mathfrak{g})$, i.e., $b\in\mathrm{Hom}%
(\mathfrak{g}$,$\mathbb{R})$ satisfies the condition $b([X,Y])=0$ for any
$X,Y\in\mathfrak{g}$. If $\Lambda^{\Xi,\mu}\neq0$, we can try to use $b$ to
cancel $\Lambda^{\Xi,\mu}$, but only for the elements that are not torsion.
Precisely, let $T_{G}$ be the torsion subgroup of $\ker\rho_{G}\simeq
H^{1}(G,\mathbb{Z})$. At the Lie algebra level, we have $\Delta^{\Xi
,\mu^{\prime}}=\Delta^{\Xi,\mu}+b\operatorname{mod}\mathbb{Z}$ and we define
$\mathfrak{T}_{G}$ as the space of $X\in\ker\exp_{G}$ such that there exists
$n\in\mathbb{N}$ with $nX\in\ker\exp_{\tilde{G}}$. Then we have $\exp
_{\tilde{G}}(\mathfrak{T}_{G})=T_{G}$ and we have the following result

\begin{lemma}
\label{lema b}If $X\in\mathfrak{T}_{G}$ and $b\in H^{1}(\mathfrak{g})$ then
$b(X)=0$.
\end{lemma}

\begin{proof}
If $X\in\mathfrak{T}_{G}$ and $nX\in\ker\exp_{\tilde{G}}$ then for any moment
map $\mu$ we have $n\Delta^{\Xi,\mu}(X)=\Delta^{\Xi,\mu}%
(nX)=0\operatorname{mod}\mathbb{Z}$. If $\mu^{\prime}=\mu+\lambda b$ then
$0=n\Delta^{\Xi,\mu^{\prime}}(X)-n\Delta^{\Xi,\mu}(X)=n\lambda
b(X)\operatorname{mod}\mathbb{Z}$ for any $\lambda\in\mathbb{R}$, and hence
$b(X)=0$.
\end{proof}

Then we have the following

\begin{theorem}
\label{Th torsion}Let $G$ be a compact and connected Lie group. Then

a) The restriction of $\Lambda^{\Xi,\mu}$ to $T_{G}$ is independent of the
moment map $\mu$.

b) There exists a $\Xi$-lift to $P$ if and only if $\Lambda^{\Xi,\mu}(\phi)=0$
for any $\phi\in T_{G}$.

c) There exists a $\Xi$-lift to $P$ if and only if $\Delta^{\Xi,\mu}(X)=0$ for
any $X\in\mathfrak{T}_{G}.$
\end{theorem}

\begin{proof}
Clearly c) and b) are equivalent and a) follows form Lemma \ref{lema b}. We
prove b). We only need to prove that if $\Lambda^{\Xi,\mu}(\phi)=0$ for any
$\phi\in T_{G}$ then there exists a $\Xi$-lift to $P$ (the converse follows
from a)). We know that $H_{1}(G,\mathbb{Z)}\simeq\ker\rho_{G}\simeq
T_{G}\oplus\mathbb{Z}^{k}$. We choose a system of generators $\phi_{1}%
\ldots,\phi_{k}$ for the free part of $\ker\rho_{G}$. If $G$ is compact,
$H^{1}(\mathfrak{g})$ it can be identified with $H^{1}(G,\mathbb{R}%
)\simeq\mathbb{R}^{k}$ by the map that assigns to $b\in H^{1}(\mathfrak{g})$
the $G$-invariant $1$-form $\alpha_{b}$ such that $\alpha_{b}|_{\mathfrak{g}%
}=b$. Hence, if $\Lambda^{\Xi,\mu}(\phi_{i})=c_{i}$, and $\phi_{i}%
=\exp_{\tilde{G}}(X_{i})$,\ then there exists $b\in H^{1}(\mathfrak{g})\simeq
H^{1}(G,\mathbb{R})$ such that $-c_{i}=\int_{\theta_{i}}\alpha_{b}%
\operatorname{mod}\mathbb{Z}=b(X_{i})\operatorname{mod}\mathbb{Z}$, where
$\theta_{i}(t)=\exp(tX_{i})$. We set $\mu^{\prime}=\mu+b$ and we have
$\Lambda^{\Xi,\mu^{\prime}}(\phi_{i})=\Delta^{\Xi,\mu^{\prime}}(X_{i}%
)=\Delta^{\Xi,\mu}(X_{i})+b(X_{i})\operatorname{mod}\mathbb{Z}=c_{i}-c_{i}=0$.
Furthermore, if the restriction of $\Lambda^{\Xi,\mu}$ to $T_{G}$ vanish, then
we have $\Lambda^{\Xi,\mu^{\prime}}(\phi)=0$ for any $\phi\in\ker\rho_{G}$.
\end{proof}

If $P_{i}\rightarrow M$, are $U(1)$-bundles with connections $\Xi_{i}$ and
moment maps $\mu_{i}$ for $\mathrm{curv}(\Xi_{i})$, $i=1,2$ then $\mu_{1}%
+\mu_{2}$ is a moment map for $\mathrm{curv}(\Xi_{1}\otimes\Xi_{2}%
)=\mathrm{curv}(\Xi_{1})+\mathrm{curv}(\Xi_{2})$ and we have $\Delta^{\Xi
_{1}\otimes\Xi_{2},\mu_{1}+\mu_{2}}=\Delta^{\Xi_{1},\mu_{1}}+\Delta^{\Xi
_{2},\mu_{2}}$. As a consequence of Theorem \ref{Th torsion} we obtain the
following classical result (e.g. see \cite{Gotay},\cite{Mundet})

\begin{corollary}
\label{power}If $G$ is a compact and connected Lie group then there exists
$r_{G}\in\mathbb{N}$ such that for any Hamiltonian action of $G$ on
$(M,\omega)$ and any pre-quantization bundle $(P,\Xi)$ for $\omega$ there
exists a $\Xi^{\otimes r_{G}}$-lift of the action to $P^{\otimes r_{G}}$.
\end{corollary}

\begin{example}
If $G=SO(n),$ $n>1$ then $H^{1}(G,\mathbb{Z})\simeq\mathbb{Z}_{2}$, and hence
$r_{SO(n)}=2$.
\end{example}

\begin{example}
If $H^{1}(G,\mathbb{Z})$ does not have torsion, then $r_{G}=1$, and any
Hamiltonian action with integral $\omega$ can be lifted to $P$.
\end{example}

\section{Equivariant pre-quantization}

In the cases in which $\Lambda^{\Xi,\mu}(X)$ only depends on $\mu$ it is easy
to obtain conditions for equivariant pre-quantization from the results of
Section \ref{Sect Lift}. In particular we have the following results:

\begin{theorem}
a) If $G$ is w-exponential, $M$ is compact and $\omega$ is symplectic, then
$(\omega,\mu)$ is\thinspace$G$-equivariant pre-quantizable if and only if
$\omega$ is integral and for any $X\in\ker\exp_{G}$ we have $\max_{M}(\mu
_{X})\in\mathbb{Z}$.

b) If $G$ and $M$ are compact and $\omega$ is symplectic, then $\omega$
is\thinspace$G$-equivariant pre-quantizable if and only if $\omega$ is
integral and for any $X\in\mathfrak{T}_{G}$ we have $\max_{M}(\mu_{X}%
)\in\mathbb{Z}$.
\end{theorem}

\begin{example}
\label{sphere}We consider the usual symplectic structure on $S^{2}$ given by
the volume form $\mathrm{vol}_{g}\in\Omega^{2}(S^{2})$ and the usual moment
map $h\colon\mathfrak{so}(3)\rightarrow\Omega^{0}(S^{2})$ given by
$h_{X}(x)=\langle\vec{v}_{X},x\rangle$, where $\vec{v}\colon\mathfrak{so}%
(3)\longrightarrow\mathbb{R}^{3}$ is the map that assigns to $X=\left(
\begin{array}
[c]{rrr}%
0 & a & b\\
-a & 0 & c\\
-b & -c & 0
\end{array}
\right)  \in\mathfrak{so}(3)$ the vector $\vec{v}_{X}=(c,-b,a).$

As $\exp X$ is a rotation of angle $||\vec{v}_{X}||$ around the axis
determined by $\vec{v}_{X}$, $\ker(\exp_{G})$ corresponds under $\vec{v}$ to
the union of the spheres of radius $2\pi k$ for $k\in\mathbb{Z}$.
If\ $X\in\ker(\exp_{G})$ and $||\vec{v}_{X}||=2\pi k$\ we have $\max_{M}%
(h_{X})=2\pi k.$

The form $\omega=\lambda\mathrm{vol}_{g}$ is integral if and only if
$\lambda=\frac{n}{4\pi}$ for $n\in\mathbb{Z}$. But in this case we have
$\max_{M}(\lambda h_{X})=\frac{nk}{2}$ that is integer for any $k$ if and only
if $n$ is even. Hence we have

\begin{proposition}
a) The form $\lambda\mathrm{vol}_{g}$ is pre-quantizable if and only if
$\lambda=\frac{n}{4\pi}$ for $n\in\mathbb{Z}$.

b) The form $\lambda\mathrm{vol}_{g}$ is $SO(3)$-equivariant pre-quantizable
if and only if $\lambda=\frac{n}{2\pi}$ for $n\in\mathbb{Z}$.
\end{proposition}

\begin{remark}
Note that if $\lambda=\frac{n}{4\pi}$, then $\lambda\mathrm{vol}_{g}$ is
$SU(2)$-equivariant pre-quantizable because $SU(2)$ is the universal cover of
$SO(3)$.
\end{remark}
\end{example}

\begin{corollary}
If $G$ is a compact and connected Lie group and we have a Hamiltonian action
of $G$ on a compact manifold $(M,\omega)$ such that $\omega$ is integral and
symplectic then we have $\max\mu_{X}(x)\in\mathbb{Q}$ for any $X\in
\mathfrak{T}_{G}$.
\end{corollary}

We also obtain the following generalization of Example 6.10 in \cite{GGK}:

\begin{theorem}
a) If $G$ is w-exponential and $M^{G}\neq\emptyset$, then $(\omega,\mu)$
is\thinspace$G$-equivariant pre-quantizable if and only if $\omega$ is
integral and for a fixed point $x\in M^{G}$ we have $\mu_{X}(x)\in\mathbb{Z}$
for any $X\in\ker\exp_{G}$.

b) If $G$ is compact and $M^{G}\neq\emptyset$, then $\omega$ is\thinspace
$G$-equivariant pre-quantizable if and only if $\omega$ is integral and for a
fixed point $x\in M^{G}$ we have $\mu_{X}(x)\in\mathbb{Z}$ for any
$X\in\mathfrak{T}_{G}$.
\end{theorem}

As a consequence of Corollary \ref{power} we obtain the following

\begin{corollary}
If $G$ is a compact and connected Lie group then there exists $r_{G}%
\in\mathbb{N}$ such that for any Hamiltonian action of $G$ on a manifold
$(M,\omega)$ with $\omega$ is integral, the form $r_{G}\cdot\omega$ is
$G$-equivariant pre-quantizable.
\end{corollary}

\begin{remark}
As commented above, if $H_{1}(G,\mathbb{Z})$ does not have torsion, then
$r_{G}=1$ and any Hamiltonian action of $G$ on a manifold $(M,\omega)$ with
$\omega$ is integral is $G$-equivariant pre-quantizable.
\end{remark}

\begin{example}
We consider the action of $S^{1}$ on $S^{2}$ by rotations around the $z\ $axis
and the form $\omega=\frac{1}{4\pi}\mathrm{vol}_{g}\in\Omega^{2}(S^{2})$ with
the moment map $\mu_{X}=\frac{1}{4\pi}h_{X\cdot M}$ for $X\in\mathbb{R}$, and
where $M=\left(
\begin{array}
[c]{rrr}%
0 & 0 & 0\\
0 & 0 & 1\\
0 & -1 & 0
\end{array}
\right)  $. For $X=2\pi k\in\ker\exp_{G}=2\pi\mathbb{Z}$ we have $\max_{M}%
(\mu_{X})=\frac{k}{2}$, and hence $(\omega,\mu)$ is not $S^{1}$-equivariant
pre-quantizable. However, as $H_{1}(S^{1},\mathbb{Z)}$ does not have torsion,
we know that $\omega$ is $S^{1}$-equivariant pre-quantizable with another
moment map $\mu^{\prime}=\mu+b$, with $b_{X}=cX$ and $c$ constant. For
example, it is enough to take $c=-\frac{1}{4\pi}$ as then we have $\max
_{M}(\mu_{X}^{\prime})=\frac{k}{2}-\frac{1}{4\pi}2\pi k=0\operatorname{mod}%
\mathbb{Z}$.
\end{example}

The case in which $\Lambda^{\Xi,\mu}(X)$ depends on $\Xi$ is more complicated
because\ it can happens that $(\omega,\mu)$ is\thinspace$G$-equivariant
pre-quantizable but $\Lambda^{\Xi,\mu}\neq0$ (see Example \ref{torus} bellow).
If $(P,\Xi)$ and $(P^{\prime},\Xi^{\prime})$ are two pre-quantizations of
$\omega$, then there exists a group homomorphism\footnote{The connection
$\Xi^{\prime}\otimes\Xi^{-1}$ on the bundle $P^{\prime}\otimes P^{-1}$ is a
flat connection and $\beta$ is the holonomy of $\Xi^{\prime}\otimes\Xi^{-1}$.}
$\beta\colon H_{1}(M,\mathbb{Z})\rightarrow\mathbb{R}/\mathbb{Z}$ such that
$\mathrm{hol}^{\Xi^{\prime}}(\gamma)=\mathrm{hol}^{\Xi}(\gamma)+\beta
([\gamma])$ for any $\gamma\in\mathcal{C}(M)$. Then we have $\Lambda
^{\Xi^{\prime},\mu}(X)=\Lambda^{\Xi,\mu}(X)+\beta([\tau_{x,X}])$. Note that if
the homology class $[\tau_{x,X}]\in H_{1}(M,\mathbb{Z})$ vanishes for any
$X\in\ker\exp_{G}$ then $\Lambda^{\Xi,\mu}(X)$ is independent of $\Xi$. This
happens for example if $H_{1}(M,\mathbb{Z})=0$.

\begin{example}
If in Example \ref{Trivial} we consider a manifold $M$ such that
$H_{1}(M,\mathbb{Z})=0$, then we obtain that $(0,\mu^{c})$ is\thinspace
$G$-equivariant pre-quantizable if and only if $2\pi c\in\mathbb{Z}$.
\end{example}

However, if $H_{1}(M,\mathbb{Z})\neq0$ the situation could be different, and
it could be possible to cancel $\Lambda^{\Xi,\mu}$ with $\beta$:

\begin{example}
\label{torus}We apply Example \ref{Trivial} to the case in which $M=T^{2}$ is
the 2-torus and $G=S^{1}$ acts on $T^{2}$ by rotations. We have $\mathrm{Hom}%
(H_{1}(T^{2},\mathbb{Z}),\mathbb{R}/\mathbb{Z})\simeq\mathbb{R}/\mathbb{Z}%
\oplus\mathbb{R}/\mathbb{Z}$, and if $X=2\pi z$\ then for any $p\in
\mathbb{R}/\mathbb{Z}$ there exists $\beta\in\mathrm{Hom}(H_{1}(T^{2}%
,\mathbb{Z}),\mathbb{R}/\mathbb{Z})$ such that $\beta([\tau_{x,X}])=p$. If we
chose $p=-2\pi zc\operatorname{mod}\mathbb{Z}$ then we have $\Lambda
^{\Xi^{\prime},\mu}(X)=\Lambda^{\Xi,\mu}(X)+\beta([\tau_{x,X}%
])=0\operatorname{mod}\mathbb{Z}$. Hence $(0,\mu^{c})$ is\thinspace
$G$-equivariant pre-quantizable for any $c\in\mathbb{\mathbb{R}}$.
\end{example}


\begin{thebibliography}{9}                                                                                                %


\bibitem {AnomaliesG}R. Ferreiro P\'{e}rez, \emph{On the geometrical
interpretation of locality in anomaly cancellation}\textbf{,} J. Geom. Phys.
133 (2018) 102--112.

\bibitem {EquiHol}---, \emph{Equivariant holonomy of }$U(1)$\emph{-bundles},
Differential Geom. Appl. \textbf{66} (2019), 1--12.

\bibitem {GGK}V. Guillemin, V. Ginzburg, Y. Karshon, \emph{Moment maps,
cobordisms, and Hamiltonian group actions}, Mathematical Surveys and
Monographs, Vol. 98, Amer. Math. Soc., Providence, RI, 2002.,

\bibitem {Gotay}M. J. Gotay and G. M. Tuynman, \emph{A symplectic analogue of
the Mostow-Palais theorem}, Symplectic geometry, groupoids and integrable
systems, MSRI Publ. 20 (ed. P. Dazord and A. Weinstein, Springer, 1989).

\bibitem {Helgason}S. Helgason, \emph{Differential Geometry, Lie Groups and
Symmetric Spaces}, Academic Press Inc. 1978.

\bibitem {Kostant}B. Kostant, \emph{Quantization and unitary representations},
Lectures in modern analysis and applications III, Lecture Notes in Math. 170
(Springer, New York, 1970), 87--208.

\bibitem {Mundet}I. Mundet i Riera, \emph{Lifts of Smooth Group Actions to
Line Bundles}, Bull. London Math. Soc., \textbf{33} (2001), 351--361.
\end{thebibliography}
\end{document}